\documentclass[11pt]{article}
\usepackage{amsfonts,amsmath}
\usepackage{hyperref,amsthm}

\newtheorem{theorem}{Theorem}[section]
\newtheorem{definition}[theorem]{Definition}
\newtheorem{proposition}[theorem]{Proposition}

\newtheorem{remark}[theorem]{Remark}

\renewenvironment{proof}[1][Proof]{ \noindent \textbf{#1: }}{$\Box$
\bigskip}

\begin{document}

\title{A Simple Reduction from a Biased Measure on
the Discrete Cube to the Uniform Measure}

\author{
Nathan Keller\thanks{The author was partially supported by the
Adams Fellowship Program of the Israeli Academy of Sciences and
Humanities and by
the Koshland Center for Basic Research.}\\
Faculty of Mathematics and Computer Science\\
Weizmann Institute of Science\\
P.O. Box~26, Rehovot 76100, Israel\\
{\tt nathan.keller@weizmann.ac.il}\\
}

\maketitle

\begin{abstract}

We show that certain statements related to the Fourier-Walsh
expansion of functions with respect to a biased measure on the
discrete cube can be deduced from the respective results for the
uniform measure by a simple reduction. In particular, we present
simple generalizations to the biased measure $\mu_p$ of the
Bonami-Beckner hypercontractive inequality, and of Talagrand's
lower bound on the size of the boundary of subsets of the discrete
cube. Our generalizations are tight up to constant factors.

\end{abstract}

\section{Introduction}

\begin{definition}
Consider the discrete cube $\{0,1\}^n$ endowed with the product measure
$\mu_p=(p \delta_{\{1\}} + (1-p) \delta_{\{0\}})^{\otimes n}$, and
let $f:\{0,1\}^n \rightarrow \mathbb{R}$. The Fourier-Walsh
expansion of $f$ with respect to the measure $\mu_p$ is the unique
expansion
\[
f = \sum_{S \subset \{1,2,\ldots,n\}} \alpha_S u_S,
\]
where for any $T \subset \{1,2,\ldots,n\}$,\footnote{Throughout
the paper, we identify elements of $\{0,1\}^n$ with subsets of
$\{1,2,\ldots,n\}$ in the natural way.}
\[
u_S(T)=\Big(-\sqrt{\frac{1-p}{p}} \Big)^{|S \cap T|}
\Big(\sqrt{\frac{p}{1-p}} \Big)^{|S \setminus T|}.
\]
In particular, for the uniform measure (i.e., $p=1/2$),
$u_S(T)=(-1)^{|S \cap T|}$. The coefficients $\alpha_S$ are
denoted by $\hat f(S)$.\footnote{Note that since the functions
$\{u_S\}_{S \subset \{1,\ldots,n\}}$ form an orthonormal basis,
the representation is indeed unique, and the coefficients are
given by the formula $\hat f(S) = \mathbb{E}_{\mu_p} [f \cdot
u_S]$ .}
\end{definition}
Properties of the Fourier-Walsh expansion are one of the main
objects of study in discrete harmonic analysis. Many of the
results in this field were obtained for the uniform measure on the
discrete cube. However, various applications (including
applications to random graph theory~\cite{Friedgut-Kalai}, to
hardness of approximation~\cite{Dinur-Safra}, and to other areas)
require consideration of a biased product measure on the discrete
cube. This led to a series of papers generalizing results from the
uniform measure $\mu_{1/2}$ to general biased measures $\mu_p$
(see,
e.g.,~\cite{Friedgut1,Hatami1,Keller-Kindler,Kindler-Thesis}). The
usual way to obtain such generalizations is to repeat the original
proof, replacing the analytic tools used in the uniform case (like
the Bonami-Beckner hypercontractive
inequality~\cite{Bonami,Beckner}) by their biased analogues. This
approach is effective, and is already considered quite standard.
However, it still requires thorough examination and adaptation of
the (sometimes complex) proofs of the results in the uniform
measure case.

\medskip

In this paper we study a simple reduction from the biased measure
$\mu_p$ to the uniform measure $\mu_{1/2}$. We note that this
reduction was already considered in several papers
(e.g.,~\cite{BKKKL}). We discuss the previous results and compare
them with our work in Section~\ref{sec:discussion}.

\medskip

Assume that $p=t/2^m$.\footnote{It is clear that there is no loss
of generality in assuming that $p$ is diadic, as the results for
general $p$ follow immediately by approximation. The exact
statement of our main result for a general $p$ is given in
Proposition~\ref{Prop:Non-Diadic}.} For any function $f:\{0,1\}^n
\rightarrow \mathbb{R}$ we define a function $Red(f)=
g:\{0,1\}^{mn} \rightarrow \mathbb{R}$ as follows: each $y \in
\{0,1\}^{mn}$ is considered as a concatenation of $n$ vectors $y^i
\in \{0,1\}^m$, and each such vector is translated to a natural
number $0 \leq Bin(y^i) < 2^m$ through its binary expansion (i.e.,
$Bin(y^i)=\sum_{j=0}^{m-1} 2^j \cdot y^i_{m-j}$). Then, for any $y
\in \{0,1\}^{mn}$,
\[
g(y)=g(y^1,y^2,\ldots,y^n):=f \left(h(y^1),h(y^2),\ldots,h(y^n) \right),
\]
where $h:\{0,1\}^m \rightarrow \{0,1\}$ is given by
\[
h(y^i) = \left\lbrace
  \begin{array}{c l}
    1, & Bin(y^i) \geq 2^m-t \\
    0, & Bin(y^i) < 2^m-t.
  \end{array}
\right.
\]
It is clear from the construction that the expectation of $g$
w.r.t. the uniform measure is equal to the expectation of $f$
w.r.t. the measure $\mu_p$. Furthermore, it was shown
in~\cite{Friedgut-Kalai} that the sum of influences\footnote{See
Section~\ref{sec:sub:lower-bound} for the definition of
influences.} of $g$ can be bounded {\it from above} by a simple
function of the sum of influences of $f$. This allows to
generalize to the biased measure statements concerning lower
bounds on the sum of influences, such as the KKL
theorem~\cite{KKL}.

\medskip

We show that for any $d$, the Fourier weight of $g$ on
the $d$-th level (i.e., $\sum_{|S|=d} \hat g(S)^2$) is bounded
{\it from below} in terms of the Fourier weight of $f$ on the
$d$-th level:
\begin{theorem}
Consider the discrete cube $\{0,1\}^n$ endowed with the product
measure $\mu_p$ for $p=t/2^m \leq 1/2$. Let $f:\{0,1\}^n
\rightarrow \mathbb{R}$, and let $Red(f)= g:\{0,1\}^{mn}
\rightarrow \mathbb{R}$ be obtained from $f$ by the construction
described above. For any $1 \leq d \leq n$,
\[
\sum_{|S|=d} \hat g(S)^2 \geq \Big(\frac{p \lfloor \log(1/p) \rfloor}{1-p} \Big)^d
\sum_{|S|=d} \hat f(S)^2,
\]
where the Fourier-Walsh coefficients of $g$ are w.r.t. the uniform
measure and the coefficients of $f$ are w.r.t. the measure
$\mu_p$. \label{Our-Theorem-Simple}
\end{theorem}
Combination of Theorem~\ref{Our-Theorem-Simple} with the results
of~\cite{Friedgut-Kalai} allows to generalize to the biased
measure statements concerning upper bounds on the Fourier weight
of the functions at any level, and statements combining such upper
bounds with lower bounds on the sum of influences. For example, we
obtain the following generalization of the Bonami-Beckner
hypercontractive inequality -- probably the most widely used tool
in discrete harmonic analysis.
\begin{proposition}\label{Prop:Beckner0}
Consider the discrete cube $\{0,1\}^n$ endowed with the product
measure $\mu_p$, $p \leq 1/2$, and let $T_{\delta}$ denote the
noise operator with rate $\delta$, defined by $T_{\delta}f=\sum_S
\delta^{|S|} \hat f(S) u_S$. For any function $f:\{0,1\}^n
\rightarrow \mathbb{R}$ and for any $0 \leq \delta \leq
\sqrt{\frac{p \lfloor \log(1/p) \rfloor}{1-p}}$,
\[
||T_{\delta}f||_2 \leq ||f||_{1+\frac{1-p}{p \lfloor \log(1/p)
\rfloor} \delta^2}.
\]
\end{proposition}

\noindent Another example is a lower bound on the size of the
vertex boundary of subsets of the discrete cube obtained by
Talagrand~\cite{Talagrand-Boundary}.
\begin{proposition}\label{Prop:Boundary0}
Consider the discrete cube $\{0,1\}^n$ endowed with the product
measure $\mu_p$, $p \leq 1/2$. There exists $\alpha>0$ such that
for any monotone subset $A \subset \{0,1\}^n$,
\begin{equation}
\mu_p(\partial A) \sum_{i=1}^n \mu_p(A_i) \geq \frac{c}{\lfloor
\log(1/p) \rfloor} \varphi \left(\mu_p(A) (1-\mu_p(A)) \right)
\psi \left(3 \lfloor \log(1/p) \rfloor \sum_{i=1}^n \mu_p(A_i)^2
\right),
\end{equation}
where $\partial A$ is the vertex boundary of $A$, $A_i$ is the
vertex boundary of $A$ in the $i$'th direction, $\varphi(x)=x^2
[\log(e/x)]^{1-\alpha}$, $\psi(x)=[\log(e/x)]^{\alpha}$, and $c>0$
is a universal constant.
\end{proposition}

\noindent Both Proposition~\ref{Prop:Beckner0} and
Proposition~\ref{Prop:Boundary0} are tight, up to constant
factors.

\medskip

We note that the main difference between the reduction technique
and the standard proof strategy presented above is that the
reduction allows to transform {\it the statement} directly to the
biased case, without considering the proof. In cases where the
proof in the uniform case is complex, like in the two propositions
above, the reduction simplifies the generalization to the biased
case significantly.

\medskip

This paper is organized as follows: The proofs of
Theorem~\ref{Our-Theorem-Simple} and of several other properties
of the simple reduction are presented in
Section~\ref{Sec:Reduction}. In Section~\ref{Sec:Applications} we
prove Propositions~\ref{Prop:Beckner0} and~\ref{Prop:Boundary0},
and present several other applications. Finally, we compare our
results with previous work and raise questions for further
research in Section~\ref{sec:discussion}.

\section{The Reduction}
\label{Sec:Reduction}

\subsection{Lower Bound on the Fourier Weight of $g=Red(f)$ on
Fixed Levels}

In this subsection we present the proof of
Theorem~\ref{Our-Theorem-Simple}. For the sake of clarity, we 
consider first the case $p=1/2^m$, where we establish a simple 
relation between the Fourier-Walsh coefficients of $g$ and the 
respective coefficients of $f$. Then we generalize the proof to any diadic
$p$, and finally we give a slightly weaker formulation of the
theorem that holds for a general $p$.

\medskip

\noindent Throughout the proof, $y=(y^1,y^2,\ldots,y^n)$ denotes an
element of $\{0,1\}^{mn}$. (Note that superscript indices, like
$y^i$, denote vectors in $\{0,1\}^m$, rather than single coordinates.) 
The functions we consider
are $f:\{0,1\}^n \rightarrow \mathbb{R}$ and
$Red(f)=g:\{0,1\}^{mn} \rightarrow \mathbb{R}$. All the
computations related to $f$ are w.r.t. the measure $\mu_p$, and
all the computations related to $g$ are w.r.t. the uniform measure
$\mu_{1/2}$.

\subsubsection{The case $p=1/2^m$}

In this case, the Fourier-Walsh coefficients of $g$ can be
expressed as a simple function of the coefficients of $f$.
\begin{proposition}
Assume that $p=1/2^m$. For any $S \subset \{1,2,\ldots,mn\}$,
denote $S_i=S \cap \{(i-1)m+1,(i-1)m+2,\ldots,im\}$. Let $S' =
\{i: |S_i|>0\}$, and denote $k=|S'|$. Then:
\[
\hat g(S) = \left(-\sqrt{\frac{p}{1-p}} \right)^k (-1)^{|S|} \hat
f(S').
\]
\label{Prop:Easy}
\end{proposition}
\begin{proof}
First, we note that since both the Fourier transform and the
construction of $g$ are linear, we have (for all $f_1,f_2,\alpha,$
and $\beta$), that if $g=Red(\alpha \cdot f_1+ \beta \cdot f_2)$,
then for any $S$,
\[
\hat g(S) = \alpha \cdot \widehat{Red(f_1)}(S) + \beta \cdot
\widehat{Red(f_2)}(S).
\]
Thus, it is sufficient to prove the assertion for the characters
$\{u_T\}_{T \subset \{1,\ldots,n\}}$, which form a basis for the
space of all functions from $\{0,1\}^n$ to $\mathbb{R}$.

\medskip

\noindent Second, we note that by the structure of the characters,
if $f=u_T$ and $g=Red(f)$, then:
\[
g(y^1,y^2,\ldots,y^n)=u_T \left(h(y^1),h(y^2),\ldots,h(y^n)
\right) = \prod_{i \in T} u_{\{i\}} (h(y^i)).
\]
Hence, we can ``decompose'' the function $g$ into functions
$\{g_i:\{0,1\}^m \rightarrow \mathbb{R}\}_{i=1,2,\ldots,n}$,
defined by $g_i(y^i)=u_{\{i\}}(h(y^i))$, and then by the
properties of the Fourier transform, we have that for any $S$,
\[
\hat g(S) = \prod_{i \in T} \hat{g_i} (S_i).
\]
Therefore, it is sufficient to prove the assertion for the
characters $u_{\{i\}}$ for $i=1,2,\ldots,n$, and the result
follows by multiplicativity.

\medskip

\noindent Third, by symmetry, it is sufficient to consider $i=1$,
and thus we may assume w.l.o.g. that $S \subset \{1,2,\ldots,m\}$.
In this case, since we clearly have $\widehat{u_{\{1\}}} (\{1\}) =
1$, the assertion of the proposition is reduced to the following:
\begin{equation}\label{Neweq2.1}
\hat g_1(S) = \left(-\sqrt{\frac{p}{1-p}} \right) \cdot
(-1)^{|S|},
\end{equation}
where $g_1=Red(u_{\{1\}})$.

\medskip

\noindent Finally, Equation~(\ref{Neweq2.1}) is obtained by simple
computation. Indeed, since $p=1/2^m$, by the definition of the
reduction we have $g_1(y^1)=-\sqrt{(1-p)/p}$ if
$y^1=(1,1,\ldots,1)$ and $g_1(y^1)=\sqrt{p/(1-p)}$ otherwise.
Since for any $S \subset \{1,2,\ldots,n\}$ we have $\mathbb{E}[u_S]=0$, 
we can write:
\begin{align}\label{Neweq2.2}
\begin{split}
\hat g_1(S) &= \mathbb{E}_{y^1 \in \{0,1\}^m} (g_1(y^1) \cdot
u_S(y^1)) = \mathbb{E}_{y^1 \in \{0,1\}^m}
((g_1(y^1)-\sqrt{p/(1-p)}) \cdot u_S(y^1)) \\
&= \frac{1}{2^m} \cdot (-\sqrt{(1-p)/p}-\sqrt{p/(1-p)}) \cdot
u_S((1,1,\ldots,1)) = p \cdot \frac{-1}{\sqrt{p(1-p)}} \cdot
(-1)^{|S|} \\
&= \left(-\sqrt{\frac{p}{1-p}} \right) \cdot (-1)^{|S|},
\end{split}
\end{align}
as asserted.
\end{proof}

\subsubsection{The case $p=t/2^m$}

In this case, the relation between the Fourier-Walsh coefficients
of $g$ and the corresponding coefficients of $f$ is a bit more
complex.
\begin{proposition}\label{Prop:Easy2}
Assume that $p=t/2^m$. For $j=1,2,\ldots,m$, let
\[
a_j(t)=\min \left(t \mbox{ mod } 2^{m-j+1}, 2^{m-j+1}-t \mbox{ mod
} 2^{m-j+1} \right).
\]
For any $S \subset \{1,2,\ldots,mn\}$, denote $S_i=S \cap
\{(i-1)m+1,(i-1)m+2,\ldots,im\}$, and $s_i=\max (S_i) - (i-1)m$.
Let $S' = \{i: |S_i|>0\}$, and denote $k=|S'|$. We have
\begin{equation}\label{Neweq2.3}
|\hat g(S)| = \left(\prod_{i \in S'} \frac{a_{s_i}(t)}{t} \right)
\cdot \left(\sqrt{\frac{p}{1-p}} \right)^k |\hat f(S')|.
\end{equation}
Furthermore, if for all $1 \leq i \leq n$, $s_i \leq \lfloor
\log(1/p) \rfloor$, then:
\begin{equation}\label{Eq2.9}
\hat g(S) = \left(-\sqrt{\frac{p}{1-p}} \right)^k (-1)^{|S|} \hat
f(S').
\end{equation}
\end{proposition}

\begin{proof}
As in the case $p=1/2^m$, it is sufficient to prove the assertion
for $f=u_{\{1\}}$ and $S \subset \{1,2,\ldots,m\}$. Denote
$Red(u_{\{1\}})=g_1$. By the definition of the reduction, we have
$g_1(y^1)=-\sqrt{(1-p)/p}$ if $Bin(y^1) \geq 2^m-t$ and
$g_1(y^1)=\sqrt{p/(1-p)}$ otherwise. Thus, a computation similar
to that given in Equation~(\ref{Neweq2.2}) shows that for any $S
\subset \{1,2,\ldots,m\}$,
\begin{align*}
\hat g_1(S) &= \sum_{\{y^1: Bin(y^1) \geq 2^m-t\}} 2^{-m} \cdot
\frac{-1}{\sqrt{p(1-p)}} \cdot u_{S}(y^1) \\
&= \frac{1}{t} \cdot \left(-\sqrt{\frac{p}{1-p}}\right) \cdot
\sum_{\{y^1: Bin(y^1) \geq 2^m-t\}} u_{S}(y^1).
\end{align*}
Thus, Equation~(\ref{Neweq2.3}) would follow once we show that
\begin{equation}\label{Neweq2.4}
\left|\sum_{\{y^1: Bin(y^1) \geq 2^m-t\}} u_{S}(y^1) \right|
=a_{s}(t),
\end{equation}
where $s=\max(S)$.

\medskip

\noindent In order to compute the left hand side of
Equation~(\ref{Neweq2.4}), we note that for any $l \in
\mathbb{N}$,
\[
\sum_{\{y^1:l \cdot 2^{m-s+1} \leq Bin(y^1) < (l+1)2^{m-s+1}\}}
u_{S}(y^1)=0.
\]
Indeed, each such sequence of $2^{m-s+1}$ consecutive values of
$y^1$ is composed of the two sequences
\[
\{y^1:2l \cdot 2^{m-s} \leq Bin(y^1) < (2l+1)2^{m-s}\} \quad
\mbox{ and } \quad \{y^1:(2l+1)2^{m-s} \leq Bin(y^1) <
(2l+2)2^{m-s}\}.
\]
Inside each of the sequences, the vectors $y^1$ differ only in
coordinates that are not included in $S$, and thus the value of
$u_S(y^1)$ is equal for all elements of the sequence. The only
difference between the sequences is in the $s$'s coordinate that
is included in $S$, and hence, the sums of $u_{S}(y^1)$ over the
sequences cancel each other. Thus, due to the cancelation we have:
\[
\sum_{\{y^1: Bin(y^1) \geq 2^m-t\}} u_{S}(y^1) = \sum_{y^1 \in V} u_{S}(y^1),
\]
where
\[
V=\{y^1:
2^m-t \leq Bin(y^1) \leq 2^m-t + \left((t-1) \mbox{ mod }
2^{m-s+1} \right)\}.
\]
Using the same argument we see that if $(t \mbox{ mod } 2^{m-j+1}
\geq 2^{m-j})$, then the value of $u_S(y^1)$ is the same for all
$y^1 \in V$, and thus,
\[
\left| \sum_{y^1 \in V} u_{S}(y^1) \right| = 2^{m-s+1} -
(t \mbox{ mod } 2^{m-s+1}).
\]
Similarly, if $(t \mbox{ mod } 2^{m-s+1} < 2^{m-s})$, then part of
the elements of the sum cancel each other, and we obtain:
\[
\left| \sum_{y^1 \in V} u_{S}(y^1) \right| = t \mbox{
mod } 2^{m-s+1}.
\]
Combining the two cases, we get:
\begin{equation}\label{Eq2.7}
\left| \sum_{\{y^1: Bin(y^1) \geq 2^m-t\}} u_{S}(y^1) \right|= 
\left| \sum_{y^1 \in V} u_{S}(y^1) \right| = \min
\left(t \mbox{ mod } 2^{m-s+1}, 2^{m-s+1}-t \mbox{ mod } 2^{m-s+1}
\right) = a_s(t),
\end{equation}
proving Equation~(\ref{Neweq2.4}).

\bigskip

\noindent If $\max(S) \leq \lfloor \log(1/p) \rfloor$, the
expression is much simpler. We note that for all $1 \leq j \leq
\lfloor \log(1/p) \rfloor$, the $j$-th coordinate of all the
vectors $y^1 \in \{0,1\}^m$ such that $Bin(y^1) \geq 2^m-t$, is
one. Thus, if $S \subset \{1,2,\ldots,\lfloor \log(1/p)
\rfloor\}$, then $u_{S}(y^1)=(-1)^{|S|}$ for all $y^1$ with
$Bin(y^1) \geq 2^m-t$. Therefore, in this case we have:
\begin{equation}\label{Eq2.8}
\sum_{\{y^1: Bin(y^1) \geq 2^m-t\}} u_{S}(y^1)=(-1)^{|S|} \cdot t,
\end{equation}
and this implies Equation~(\ref{Eq2.9}), completing the proof of
Proposition~\ref{Prop:Easy2}.
\end{proof}

\bigskip

\noindent Theorem~\ref{Our-Theorem-Simple} follows immediately
from Proposition~\ref{Prop:Easy2}. Indeed, setting:
\begin{align*}
A_{\{i_1,i_2,\ldots,i_k\}} = \{S \subset \{1,\ldots,n\}:(|S|=k)
\wedge (S'=\{i_1,\ldots,i_k\}) \wedge (\forall i \in S', s_i \leq
\lfloor \log(1/p) \rfloor) \},
\end{align*}
we have (by Equation~(\ref{Eq2.9})), that:
\begin{align*}
\sum_{|S|=k} \hat g(S)^2 &\geq \sum_{\{i_1,i_2,\ldots,i_k\}
\subset \{1,\ldots,n\}} \left( \sum_{S \in A_{\{i_1,i_2,\ldots,i_k\}}} \hat g(S)^2 \right) \\
&= \sum_{\{i_1,i_2,\ldots,i_k\} \subset \{1,\ldots,n\}} (\lfloor
\log(1/p) \rfloor)^k \left( \frac{p}{1-p} \right)^k \hat
f(\{i_1,i_2,\ldots,i_k\})^2
\\
&= \left( \frac{p \lfloor \log(1/p) \rfloor}{1-p} \right)^k
\sum_{|S'|=k} \hat f(S')^2,
\end{align*}
as asserted.

\subsubsection{The Non-diadic Case}

In the case of a non-diadic $p$, we can use approximation to get a
slightly weaker variant of Theorem~\ref{Our-Theorem-Simple}.
\begin{proposition}\label{Prop:Non-Diadic}
Consider the discrete cube $\{0,1\}^n$ endowed with the product
measure $\mu_p$, $p \leq 1/2$. For any function $f:\{0,1\}^n
\rightarrow \mathbb{R}$, and for any $\epsilon>0$, there exists $m
\in \mathbb{N}$ and a function $g:\{0,1\}^{mn} \rightarrow
\mathbb{R}$, such that:
\begin{itemize}
\item $|\mathbb{E}[g]-\mathbb{E}[f]|<\epsilon$, and

\item For all $1 \leq d \leq n$,
\[
\sum_{|S|=d} \hat g(S)^2 \geq \Big(\frac{p \lfloor \log(1/p) \rfloor}{1-p} \Big)^d
\sum_{|S|=d} \hat f(S)^2 - \epsilon,
\]
\end{itemize}
where the expectation and the Fourier-Walsh coefficients of $g$ are w.r.t. the uniform
measure and the expectation and the coefficients of $f$ are w.r.t. the measure
$\mu_p$.
\end{proposition}
\begin{proof}
For any function $f:\{0,1\}^n \rightarrow \mathbb{R}$, the maps $p
\mapsto \mathbb{E}_{\mu_p}(f)$ and $p \mapsto \sum_{|S|=d} \hat
f_{\mu_p}(S)^2$ (where $\hat f_{\mu_p}(S)$ denotes coefficients
w.r.t. the measure $\mu_p$) are uniformly continuous (as functions
of $p$) in $[0,1]$. Therefore, for any $\epsilon>0$, we can choose
a diadic $p'=t/2^m$ close enough to $p$ such that the function
$g:\{0,1\}^{mn} \rightarrow \mathbb{R}$ constructed from $f$ by
the reduction procedure w.r.t. the measure $\mu_{p'}$ will
satisfy the conditions of the proposition.
\end{proof}

\subsection{Upper Bound on the Influences of $g$}
\label{sec:sub:lower-bound}

Before we turn to the applications, we present the result of
Friedgut and Kalai~\cite{Friedgut-Kalai} mentioned in the
introduction that allows to bound the sum of influences of $g$
from above in terms of the influences of $f$, and obtain a variant
of that result that allows to bound the sum of squares of the
influences. We note that while the result of~\cite{Friedgut-Kalai}
generalizes to the continuous cube with the Lebesgue measure (as
shown in~\cite{Keller-Cont-Influence}), the proof of the bound on
the squares of influences holds only in the discrete setting.
\begin{definition}
Consider the discrete cube $\{0,1\}^n$ endowed with the product
measure $\mu_p$. Let $f:\{0,1\}^n \rightarrow \{0,1\}$. For $1
\leq i \leq n$, the influence of the $i$-th coordinate on $f$ is
\[
I_i(f)=\Pr_{x \sim \mu_p} [f(x) \neq f(x \oplus e_i)],
\]
where $x\oplus e_i$ denotes the vector obtained from $x$ by
replacing $x_i$ by $1-x_i$ and leaving the other coordinates
unchanged.

\medskip

\noindent For a set $A \subset \{0,1\}^n$, we define
$I_i(A)=I_i(1_A)$.
\end{definition}
\begin{theorem}[Friedgut and Kalai]\label{Thm:Lower-Bound}
Consider the discrete cube $\{0,1\}^n$ endowed with the product
measure $\mu_p$, for $p=t/2^m \leq 1/2$. Let $f:\{0,1\}^n
\rightarrow \{0,1\}$, and let $g:\{0,1\}^{mn} \rightarrow \{0,1\}$
be obtained from $f$ by the construction described in the
introduction. Then
\[
\sum_{i=1}^{mn} I_i(g) \leq 6 p \lfloor \log(1/p) \rfloor
\sum_{i=1}^n I_i(f),
\]
where the influences on $f$ are w.r.t. $\mu_p$, and the influences
on $g$ are w.r.t. $\mu_{1/2}$.
\end{theorem}

\begin{proposition} \label{Prop:Lower-Bound}
Let $f,g$ be as defined in Theorem~\ref{Thm:Lower-Bound}. Then
\[
\sum_{i=1}^{mn} I_i(g)^2 \leq 12 p^2 \lfloor \log(1/p) \rfloor
\sum_{i=1}^n I_i(f)^2.
\]
\end{proposition}
\noindent The proof is essentially the same as the proof of
Theorem~\ref{Thm:Lower-Bound} given in~\cite{Friedgut-Kalai}. For
the sake of completeness we present it here.

\begin{proof}
For $1 \leq i \leq n$, consider the influences of the coordinates
$(i-1)m+1,\ldots,im$ on $g$. It is easy to see that\footnote{We
note that this part of the argument does not hold for the Lebesgue
measure on the continuous cube. It does hold for any single fiber
(w.r.t. the expectation on that fiber), but unlike the case of
influences where the fibers can be combined using Fubini's
theorem, in the case of squares of influences the fibers cannot be
combined.}
\[
I_{(i-1)m+j}(g) \leq \left\lbrace
  \begin{array}{c l}
    2p \cdot I_i(f), & j \leq \lfloor \log(1/p) \rfloor, \\
    2^{-j+2} \cdot I_i(f), & j>\lfloor \log(1/p) \rfloor.
  \end{array}
\right.
\]
Thus,
\begin{align*}
\sum_{j=1}^m I_{(i-1)m+j}(g)^2 &\leq \left( \lfloor \log(1/p)
\rfloor \cdot 4p^2 + \sum_{j=\lfloor \log(1/p) \rfloor+1}^m
2^{-2j+4} \right) I_i(f)^2 \\ &\leq (4 p^2 \lfloor \log(1/p)
\rfloor + 8p^2) I_i(f)^2 \leq 12 p^2 \lfloor \log(1/p) \rfloor
I_i(f)^2.
\end{align*}
Summing over $i$ completes the proof.
\end{proof}

\medskip

\noindent Since for a fixed function $f:\{0,1\}^n \rightarrow
\{0,1\}$, the maps $p \mapsto \sum_i I^p_i(f)$ and $p \mapsto
\sum_i I^p_i(f)^2$ (where $I^p_i$ denotes influence w.r.t.
$\mu_p$) are uniformly continuous as function of $p$ in $[0,1]$,
we immediately get the following:
\begin{proposition}\label{Prop:Lower-Bound-Approximation}
Consider the discrete cube $\{0,1\}^n$ endowed with the product
measure $\mu_p$, $p \leq 1/2$. For any function $f:\{0,1\}^n
\rightarrow \{0,1\}$, and for any $\epsilon>0$, there exists $m
\in \mathbb{N}$ and a function $g:\{0,1\}^{mn} \rightarrow
\mathbb{R}$, such that:
\begin{itemize}
\item $|\mathbb{E}[g]-\mathbb{E}[f]|<\epsilon$,

\item $\sum_{i=1}^{mn} I_i(g) \leq 6 p \lfloor \log(1/p) \rfloor
\sum_{i=1}^n I_i(f) + \epsilon$, and

\item $\sum_{i=1}^{mn} I_i(g)^2 \leq 12 p^2 \lfloor \log(1/p)
\rfloor \sum_{i=1}^n I_i(f)^2 + \epsilon$,
\end{itemize}
where the expectation and the influences of $g$ are w.r.t. the
uniform measure and the expectation and the influences of $f$ are
w.r.t. the measure $\mu_p$.
\end{proposition}
\noindent Furthermore, it is clear that $g$ can be chosen such
that it will satisfy the assertions of
Propositions~\ref{Prop:Non-Diadic}
and~\ref{Prop:Lower-Bound-Approximation} simultaneously.

\begin{remark} \label{Rem:Lower1}
We note that using Proposition~\ref{Prop:Easy2}, the proof of
Proposition~\ref{Prop:Lower-Bound} can be easily adapted to show
that
\begin{equation} \label{Eq:Lower1}
\sum_{i=1}^{mn} \hat g(\{i\})^2 \leq 3 \frac{p \lfloor \log(1/p)
\rfloor}{1-p} \sum_{i=1}^n \hat f(\{i\})^2,
\end{equation}
proving the tightness of Theorem~\ref{Our-Theorem-Simple} for
$d=1$ up to a multiplicative constant.
\end{remark}

\section{Applications}
\label{Sec:Applications}

In this section we apply Theorem~\ref{Our-Theorem-Simple} to
obtain simple generalizations to the biased measure $\mu_p$ of
several results:
\begin{itemize}
\item The Bonami-Beckner hypercontractive
inequality~\cite{Bonami,Beckner},

\item A relation between the size of the vertex boundary of a
monotone\footnote{A function $f:\{0,1\}^n \rightarrow \mathbb{R}$
is monotone if for all $x,y \in [0,1]^n$, $\forall i (x_i \geq
y_i) \Longrightarrow f(x) \geq f(y)$. A subset $A \subset
\{0,1\}^n$ is monotone if its characteristic function $1_A$ is
monotone.} subset of the discrete cube and its influences,
obtained by Talagrand~\cite{Talagrand-Boundary},

\item An upper bound on the $d$-th level Fourier-Walsh
coefficients of a monotone function in terms of its influences,
obtained by Talagrand~\cite{Talagrand-Correlation} for $d=2$ and
generalized by Benjamini et al.~\cite{BKS} to any $d$, and

\item A lower bound on the correlation between monotone families,
obtained by Talagrand~\cite{Talagrand-Correlation}.
\end{itemize}

\noindent Throughout this section, we assume for the sake of
simplicity that $p$ is diadic (i.e., $p=t/2^m$), and satisfies $p
\leq 1/2$. All the proofs generalize immediately to a non-diadic
$p$ by choosing a diadic $p'$ ``close enough'' to $p$, replacing
Theorems~\ref{Our-Theorem-Simple} and~\ref{Thm:Lower-Bound} by
Propositions~\ref{Prop:Non-Diadic}
and~\ref{Prop:Lower-Bound-Approximation} (respectively), and
considering the limit as $\epsilon \rightarrow 0$. The case
$p>1/2$ also follows immediately by considering a variant of the
{\it dual function} defined as
\[
f'(x_1,x_2,\ldots,x_n)=f(1-x_1,1-x_2,\ldots,1-x_n),
\]
and noting that for any $S \subset \{1,\ldots,n\}$, we have $\hat
{f'}(S)=(-1)^{|S|} \hat f(S)$, where the Fourier-Walsh coefficients
of $f$ are w.r.t. $\mu_p$ and the coefficients of $f'$ are w.r.t.
$\mu_{1-p}$.

\medskip

\noindent For a function $f:\{0,1\}^n \rightarrow \mathbb{R}$, the
function $Red(f)=g:\{0,1\}^{mn} \rightarrow \mathbb{R}$ denotes the
function obtained from $f$ by the reduction procedure described in the
introduction. All the computations related to $f$ are w.r.t. the
measure $\mu_p$, and all the computations related to $g$ are
w.r.t. the uniform measure.

\subsection{The Bonami-Beckner Hypercontractive Inequality}
\label{sec:sub:Beckner}

One of the main tools in discrete harmonic analysis is the
Bonami-Beckner hypercontractive inequality~\cite{Bonami,Beckner}.
The inequality considers a special operator called the {\it noise
operator} that has a simple description in terms of the
Fourier-Walsh expansion.
\begin{definition}
For a function $f:\{0,1\}^n \rightarrow \mathbb{R}$ with
Fourier-Walsh expansion $f=\sum_S \hat f(S) u_S$, the application
of the noise operator with rate $\delta$ to $f$ is
\[
T_{\delta}f=\sum_S \delta^{|S|} \hat f(S) u_S.
\]
\end{definition}
\begin{theorem}[Bonami,Beckner]
Consider the discrete cube $\{0,1\}^n$ endowed with the uniform
measure. For any function $f:\{0,1\}^n \rightarrow \mathbb{R}$ and
for any $0 \leq \delta \leq 1$,
\[
||T_{\delta}f||_2 \leq ||f||_{1+\delta^2}.
\]
\end{theorem}
\noindent Using Theorem~\ref{Our-Theorem-Simple}, we get the
following generalization to a biased measure $\mu_p$:
\begin{proposition}\label{Prop:Beckner-our}
Consider the discrete cube $\{0,1\}^n$ endowed with the product
measure $\mu_p$, $p \leq 1/2$. For any function $f:\{0,1\}^n
\rightarrow \mathbb{R}$ and for any $0 \leq \delta \leq
\sqrt{\frac{p \lfloor \log(1/p) \rfloor}{1-p}}$,
\[
||T_{\delta}f||_2 \leq ||f||_{1+\frac{1-p}{p \lfloor \log(1/p)
\rfloor} \delta^2}.
\]
\end{proposition}
\begin{proof}
We have
\begin{align*}
||T_{\delta}f||_2^2 &= \sum_S \delta^{2|S|} \hat f(S)^2 =
\sum_{d=0}^n \delta^{2d} \sum_{|S|=d} \hat f(S)^2  \leq
\sum_{d=0}^n \delta^{2d} \cdot \left( \frac{1-p}{p \lfloor
\log(1/p) \rfloor} \right)^d \sum_{|S|=d} \hat g(S)^2 \\
&\leq \sum_{d=0}^{mn} \left( \delta \sqrt{\frac{1-p}{p \lfloor
\log(1/p) \rfloor}} \right)^{2d} \sum_{|S|=d} \hat g(S)^2 =
||T_{\delta \sqrt{\frac{1-p}{p \lfloor \log(1/p) \rfloor}}}
g||_2^2 \\ &\leq ||g||_{1+\frac{1-p}{p \lfloor \log(1/p) \rfloor}
\delta^2}^2 = ||f||_{1+\frac{1-p}{p \lfloor \log(1/p) \rfloor}
\delta^2}^2.
\end{align*}
The first and the third equalities follow from the Parseval
identity, the first inequality follows from
Theorem~\ref{Our-Theorem-Simple}, the third inequality follows
from the Bonami-Beckner inequality (for the uniform measure), and
the last equality follows since by the construction of $g$, we
have $||f||_q=||g||_q$ for any norm $q$.
\end{proof}

\noindent By the duality of $L^p$ norms,
Proposition~\ref{Prop:Beckner-our} implies:
\begin{proposition} \label{Prop:Beckner-our2}
Consider the discrete cube $\{0,1\}^n$ endowed with the product
measure $\mu_p$, $p \leq 1/2$. If a function $f:\{0,1\}^n
\rightarrow \mathbb{R}$ satisfies $\hat f(S)=0$ for all $|S|>d$,
then for any $q \geq 2$,
\[
||f||_q \leq \left( \frac{1-p}{p \lfloor \log(1/p) \rfloor} \cdot
(q-1) \right)^{d/2} ||f||_2.
\]
\end{proposition}
\noindent We note that the problem of finding the optimal
hypercontractivity constant for a biased measure $\mu_p$, i.e.,
finding the minimal value $C_{p,q}$ for which in the assumptions
of Proposition~\ref{Prop:Beckner-our2},
\[
||f||_q \leq (C_{p,q})^{d/2} ||f||_2,
\]
was studied in various works, in several different contexts. Partial
results were obtained by Talagrand~\cite{Talagrand1},
Friedgut~\cite{Friedgut1} and Kindler~\cite{Kindler-Thesis} and
applied in the study of Boolean functions. The optimal value of
$C_{p,q}$ is attributed to Rothaus (unpublished), was stated
without proof in lecture notes of Higuchi and
Yoshida~\cite{Higuchi}, and given with proof by Diaconis and
Saloff-Coste~\cite{Diaconis} in the context of the logarithmic
Sobolev inequality. The formulation we use in
Proposition~\ref{Prop:Beckner-our2} was used by
Oleszkiewicz~\cite{Oles} in the context of the Khinchine-Kahane
inequality. For $q \geq \ln(1/p)$ (which is usually the case in
applications), the value of $C_{p,q}$ obtained in
Proposition~\ref{Prop:Beckner-our2} matches the optimal value
given in Theorem~2.1 of~\cite{Oles}, up to a multiplicative
constant.

\subsection{Relation Between the Influences and the Size of the Boundary}
\label{sec:sub:boundary}

\begin{definition}
Consider the discrete cube $\{0,1\}^n$ endowed with the product
measure $\mu_p$. For a monotone set $A \subset \{0,1\}^n$ and $1
\leq i \leq n$, let
\[
A_i=\{x \in A: x \oplus e_i \not \in A \}.\footnote{Note that this
definition is closely related to the notion of influences, as for
any monotone set $A$, we have $\mu_p(A_i) = p \cdot I_i(A)$, where
the influence is w.r.t. the measure $\mu_p$.}
\]
The vertex boundary of $A$ is
\[
\partial A = \bigcup_{i=1}^n A_i = \{x \in A: \exists (1 \leq i \leq n), x \oplus e_i \not \in A \}.
\]
\end{definition}

\noindent In~\cite{Margulis}, Margulis proved that for subsets of
the discrete cube endowed with the uniform measure, the size of
the boundary and the sum of influences cannot be small
simultaneously. For monotone subsets of the discrete cube,
Talagrand~\cite{Talagrand-Boundary} gave the following precise
form to this statement:
\begin{theorem}[Talagrand]\label{Thm:Talagrand-Boundary}
Consider the discrete cube $\{0,1\}^n$ endowed with the uniform
measure $\mu_{1/2}$. There exists $\alpha>0$ such that for any
monotone subset $A \subset \{0,1\}^n$,
\[
\mu_{1/2}(\partial A) \sum_{i=1}^n \mu_{1/2}(A_i) \geq c \varphi
\left(\mu_{1/2}(A) (1-\mu_{1/2}(A)) \right) \psi
\left(\sum_{i=1}^n \mu_{1/2}(A_i)^2 \right),
\]
where $\varphi(x)=x^2 [\log(e/x)]^{1-\alpha}$,
$\psi(x)=[\log(e/x)]^{\alpha}$, and $c>0$ is a universal constant.
\end{theorem}

\noindent Using Proposition~\ref{Prop:Lower-Bound}, we get the
following generalization to the measure $\mu_p$:
\begin{proposition}\label{Prop:Boundary}
Consider the discrete cube $\{0,1\}^n$ endowed with the product
measure $\mu_p$, $p \leq 1/2$. There exists $\alpha>0$ such that
for any monotone subset $A \subset \{0,1\}^n$,
\begin{equation}\label{Eq:Boundary0}
\mu_p(\partial A) \sum_{i=1}^n \mu_p(A_i) \geq \frac{c}{\lfloor
\log(1/p) \rfloor} \varphi \left(\mu_p(A) (1-\mu_p(A)) \right)
\psi \left(3 \lfloor \log(1/p) \rfloor \sum_{i=1}^n \mu_p(A_i)^2
\right),
\end{equation}
where $\varphi(x)=x^2 [\log(e/x)]^{1-\alpha}$,
$\psi(x)=[\log(e/x)]^{\alpha}$, and $c>0$ is a universal constant.
\end{proposition}

\begin{proof}
Denote $f=1_A$, construct $g:\{0,1\}^{mn} \rightarrow \{0,1\}$ as
described in the introduction, and let $B=\{y \in
\{0,1\}^{mn}:g(y)=1\}$. It is easy to see that $\mu_{1/2}(\partial
B) \leq \mu_p(\partial A)$. Indeed, to any $y=(y^1,y^2,\ldots,y^n)
\in \{0,1\}^{mn}$ we can attach $H(y) \in \{0,1\}^n$ by setting
$H(y)_i=1$ if $y^i \neq (0,0,\ldots,0)$ and $H(y)_i=0$ otherwise,
and then it is clear by the construction of $g$ that $(y \in
\partial(B)) \Rightarrow (H(y) \in \partial(A))$. Since for any $x
\in A$ we have $\mu_{1/2}(\{y \in B: H(y)=x\}) =\mu_p(x)$, it
follows that $\mu_{1/2}(\partial B) \leq \mu_p(\partial A)$.

\medskip

\noindent Furthermore, it follows from
Theorem~\ref{Thm:Lower-Bound} that
\[
\sum_{j=1}^{mn} \mu_{1/2}(B_j) \leq 3 \lfloor \log(1/p) \rfloor
\sum_{i=1}^n \mu_p(A_i).
\]
Therefore,
\begin{equation}\label{Eq:Boundary1}
\mu_{1/2}(\partial B) \sum_{j=1}^{mn} \mu_{1/2}(B_j) \leq 3
\lfloor \log(1/p) \rfloor \mu_p(\partial A) \sum_{i=1}^n
\mu_p(A_i).
\end{equation}
On the other hand, it follows from
Proposition~\ref{Prop:Lower-Bound} that
\[
\sum_{j=1}^{mn} \mu_{1/2}(B_j)^2 \leq 3 \lfloor \log(1/p) \rfloor
\sum_{i=1}^{n} \mu_{p}(A_i)^2.
\]
Since by the construction of $B$, we have $\mu_{1/2}(B)=\mu_p(A)$,
and since the function $\psi(x)=\log(e/x)^{\alpha}$ is monotone
decreasing, it follows that
\begin{equation}\label{Eq:Boundary2}
\varphi \left(\mu_{1/2}(B) (1-\mu_{1/2}(B)) \right) \psi
\left(\sum_{j=1}^{mn} \mu_{1/2}(B_j)^2 \right) \geq \varphi
\left(\mu_p(A) (1-\mu_p(A)) \right) \psi \left( 3 \lfloor
\log(1/p) \rfloor \sum_{i=1}^n \mu_p(A_i)^2 \right).
\end{equation}
Combining Equations~(\ref{Eq:Boundary1}) and~(\ref{Eq:Boundary2})
with Theorem~\ref{Thm:Talagrand-Boundary}, we get
\begin{align*}
\mu_p(\partial A) \sum_{i=1}^n \mu_p(A_i) &\geq \frac{1}{3 \lfloor
\log(1/p) \rfloor} \mu_{1/2}(\partial B) \sum_{j=1}^{mn}
\mu_{1/2}(B_j)
\\
&\geq \frac{c}{3 \lfloor \log(1/p) \rfloor} \varphi
\left(\mu_{1/2}(B) (1-\mu_{1/2}(B)) \right) \psi
\left(\sum_{j=1}^{mn} \mu_{1/2}(B_j)^2 \right)
\\
&\geq \frac{c'}{ \lfloor \log(1/p) \rfloor} \varphi \left(\mu_p(A)
(1-\mu_p(A)) \right) \psi \left( 3 \lfloor \log(1/p) \rfloor
\sum_{i=1}^n \mu_p(A_i)^2 \right),
\end{align*}
as asserted.
\end{proof}

\medskip

\noindent We show now that Proposition~\ref{Prop:Boundary} is
tight by considering a balanced threshold function on the biased
discrete cube. Consider the measure $\mu_p$ on $\{0,1\}^n$, and
let
\[
A=\{x \in \{0,1\}^n : \sum_{i=1}^n x_i > \lfloor np \rfloor\}.
\]
It is well-known that $\mu_p(A) = \Theta(1)$. We have
\[
\partial A = \{x \in \{0,1\}^n : \sum_{i=1}^n x_i = \lfloor np \rfloor +1\},
\]
and hence it can be shown using Stirling's formula that
\[
\mu_p(\partial A) \approx \frac{1}{\sqrt{2 \pi n p(1-p)}}.
\]
Similarly,
\[
A_i = \{x \in \{0,1\}^n : \left(\sum_{i=1}^n x_i = \lfloor np
\rfloor +1 \right) \wedge (x_i=1) \},
\]
and thus,
\[
\mu_p(A_i) \approx \sqrt{\frac{p}{2 \pi n (1-p)}}.
\]
Therefore,
\[
\mbox{ L.h.s. of~(\ref{Eq:Boundary0}) } = \mu_p(\partial A)
\sum_{i=1}^n \mu_p(A_i) \approx n \cdot \frac{1}{\sqrt{2 \pi n
p(1-p)}} \cdot \sqrt{\frac{p}{2 \pi n (1-p)}} = \Theta(1).
\]
In the right hand side, we have $\sum_{i=1}^n \mu_p(A_i)^2 \approx
\frac{p}{2 \pi (1-p)}$, and thus,
\[
\psi \left( 3 \lfloor \log(1/p) \rfloor \sum_{i=1}^n \mu_p(A_i)^2
\right) \approx c(\log(1/p))^{\alpha}.
\]
Since $\varphi \left(\mu_p(A)(1-\mu_p(A)) \right) = \Theta(1)$, it
follows that
\[
\mbox{ R.h.s. of~(\ref{Eq:Boundary0}) } \approx \frac{c}{\lfloor
\log(1/p) \rfloor} \cdot (\log(1/p))^{\alpha} = c'
(\log(1/p))^{\alpha-1}.
\]
Therefore, if Theorem~\ref{Thm:Talagrand-Boundary} holds for
$\alpha=1$ as conjectured by Talagrand~\cite{Talagrand-Boundary},
then the assertion of Proposition~\ref{Prop:Boundary} is tight, up
to constant factors.\footnote{We note that we are not aware of a 
non-trivial example demonstrating the tightness of Talagrand's 
conjectured assertion for the uniform measure.}

\medskip

\noindent We note that in~\cite{Talagrand-Boundary}, Talagrand
remarked that Theorem~\ref{Thm:Talagrand-Boundary} can be
generalized to a biased measure, but he refrained from doing so
since it would have required ``to write in the case $p \neq 1/2$
the proof of the lengthy Lemma 2.2 below''. Our proof shows that
the statement of the theorem can be generalized to the biased
measure ``automatically'', avoiding the need to repeat the proof.

\subsection{Upper Bound on the $d$-th Level Fourier-Walsh
Coefficients} \label{sec:sub:Talagrand-Fourier}

In~\cite{Talagrand-Correlation}, Talagrand obtained the following
upper bound on the second level Fourier-Walsh coefficients of a
monotone Boolean function in terms of its
influences.\footnote{Actually, Talagrand proved a decoupled
version of the theorem, bounding $\sum_{|S|=2} \hat f(S) \hat
g(S)$ for monotone functions $f,g$. Our generalization holds for
the decoupled version as well.}
\begin{theorem}[Talagrand]\label{Thm:Talagrand-Bound}
Consider the discrete cube $\{0,1\}^n$ endowed with the uniform
measure. Let $f:\{0,1\}^n \rightarrow \{0,1\}$ be a monotone
function. Then
\[
\sum_{|S|=2} \hat f(S)^2 \leq c \sum_{i=1}^n I_i(f)^2 \cdot \log
\left(\frac{e}{\sum_{i=1}^n I_i(f)^2} \right),
\]
where $c$ is a universal constant.
\end{theorem}

\noindent Talagrand used Theorem~\ref{Thm:Talagrand-Bound} as a
central lemma in proving Theorem~\ref{Thm:Talagrand-Boundary} and
Theorem~\ref{Thm:Talagrand-Correlation} (see below).
In~\cite{BKS}, Benjamini et al. generalized Talagrand's theorem to
bound the $d$-th level coefficients:
\begin{theorem}[Benjamini, Kalai, and Schramm] \label{Thm:BKS-lemma}
Consider the discrete cube
$\{0,1\}^n$ endowed with the uniform measure. Let $f:\{0,1\}^n
\rightarrow \{0,1\}$ be a monotone function. Then for any $d \geq
2$,
\[
\sum_{|S|=d} \hat f(S)^2 \leq C_d \sum_{i=1}^n I_i(f)^2 \cdot
\log^{d-1} \left(\frac{e}{\sum_{i=1}^n I_i(f)^2} \right),
\]
where $C_d$ is a constant depending only on $d$.
\end{theorem}

\noindent Theorem~\ref{Thm:BKS-lemma} was used by Benjamini et al.
to show a qualitative relation between noise sensitivity of
Boolean functions and their influences, and has applications in
the study of percolation.

\medskip

\noindent Using Theorem~\ref{Our-Theorem-Simple}, we obtain the
following generalization of Theorem~\ref{Thm:BKS-lemma} to the
biased measure:
\begin{proposition}\label{Prop:BKS-lemma}
Consider the discrete cube $\{0,1\}^n$ endowed with the measure
$\mu_p$, $p \leq 1/2$. Let $f:\{0,1\}^n \rightarrow \{0,1\}$ be a
monotone function. Then for any $d \geq 2$,
\[
\sum_{|S|=d} \hat f(S)^2 \leq C_d \cdot \left( \frac{1-p}{p
\lfloor \log(1/p) \rfloor} \right)^{d-1} \cdot p(1-p) \cdot
\sum_{i=1}^n I_i(f)^2 \cdot \log^{d-1} \left(\frac{e}{p^2 \lfloor
\log(1/p) \rfloor \sum_{i=1}^n I_i(f)^2} \right),
\]
where $C_d$ is a constant depending only on $d$.
\end{proposition}
\begin{proof}
We have
\begin{align*}
\sum_{|S|=d} \hat f(S)^2 &\leq \left( \frac{1-p}{p \lfloor
\log(1/p) \rfloor} \right)^d \sum_{|S|=d} \hat g(S)^2 \\
&\leq \left( \frac{1-p}{p \lfloor \log(1/p) \rfloor} \right)^d
\cdot C_d \cdot \sum_{i=1}^{mn} I_i(g)^2 \cdot \log^{d-1}
\left(\frac{e}{\sum_{i=1}^{mn} I_i(g)^2} \right) \\
&\leq \left( \frac{1-p}{p \lfloor \log(1/p) \rfloor} \right)^{d-1}
\cdot C'_d \cdot p(1-p) \cdot \sum_{i=1}^n I_i(f)^2 \cdot
\log^{d-1} \left(\frac{e}{p^2 \lfloor \log(1/p) \rfloor} \cdot
\sum_{i=1}^n I_i(f)^2 \right).
\end{align*}
The first inequality follows from
Theorem~\ref{Our-Theorem-Simple}, the second follows from
Theorem~\ref{Thm:BKS-lemma}, and the third one follows from
combination of Theorem~\ref{Our-Theorem-Simple} and
Proposition~\ref{Prop:Lower-Bound}.
\end{proof}

\medskip

\noindent We note that in a recent paper~\cite{Keller-Kindler},
Theorem~\ref{Thm:BKS-lemma} was generalized to non-monotone
functions, and the dependence of $C_d$ on $d$ was determined. Our
generalization applies to these results as well, and the resulting
upper bound for the measure $\mu_p$ is slightly weaker than the
bound obtained in~\cite{Keller-Kindler} by adaptation of the
entire proof to the biased measure.

\subsection{Lower Bound on the Correlation Between Monotone
Families} \label{sec:sub:correlation}

In~\cite{Talagrand-Correlation}, Talagrand obtained a lower bound
on the correlation between monotone families, improving over the
classical Harris-Kleitman correlation inequality:
\begin{theorem}[Talagrand] \label{Thm:Talagrand-Correlation}
Consider the discrete cube $\{0,1\}^n$ endowed with the uniform
measure $\mu$. For any pair of monotone families $A,B \subset
\{0,1\}^n$,
\[
\mu(A \cap B)-\mu(A)\mu(B) \geq c \varphi \left(\sum_{i=1}^n
I_i(A) I_i(B) \right),
\]
where $\varphi(x)=x/\log(e/x)$, and $c$ is a universal constant.
\end{theorem}

\noindent Using Theorem~\ref{Our-Theorem-Simple}, we
get the following generalization:
\begin{proposition}\label{Prop:Correlation}
Consider the discrete cube $\{0,1\}^n$ endowed with the measure
$\mu_p$, $p \leq 1/2$. For any pair of monotone families $A,B
\subset \{0,1\}^n$,
\[
\mu_p(A \cap B)-\mu_p(A)\mu_p(B) \geq c \varphi \left(\lfloor
\log(1/p) \rfloor \sum_{i=1}^n I_i(A) I_i(B) \right),
\]
where the influences are w.r.t. the measure $\mu_p$,
$\varphi(x)=x/\log(e/x)$, and $c$ is a universal constant.
\end{proposition}

\noindent Since for any monotone function $f$, $\hat f(\{i\}) =
-\sqrt{p(1-p)} I_i(f)$, the proposition follows immediately from
Theorem~\ref{Our-Theorem-Simple} using the monotonicity of the
function $\varphi(x)$ in $(0,1]$.

\medskip

\noindent We note that a slightly better result was obtained
in~\cite{Keller-Thesis} by transforming Talagrand's proof of
Theorem~\ref{Thm:Talagrand-Correlation} to the biased measure
$\mu_p$.

\section{Discussion}
\label{sec:discussion}

The standard reduction discussed in this paper was considered and
used in numerous papers. It was first suggested in~\cite{BKKKL},
as a transformation from a function $f:[0,1]^n \rightarrow
\{0,1\}$ to a function $g:\{0,1\}^{mn} \rightarrow \{0,1\}$. The
reduction was used there to generalize the KKL theorem~\cite{KKL}
to the continuous cube endowed with the Lebesgue measure by
obtaining a lower bound on the maximal influence.
In~(\cite{Friedgut-Kalai}, Theorem~3.1) the reduction was adapted
to the biased measure on the discrete cube (yielding
Theorem~\ref{Thm:Lower-Bound} above), and used to prove that any
monotone graph property has a sharp threshold. In~\cite{Friedgut2}
the technique of~\cite{BKKKL} was simplified and extended to the
context of influences toward one and zero.
In~\cite{Bollobas-Riordan} the reduction was applied to a product
measure on $\{0,1,2\}^n$ and used in computing the critical
probability of random Voronoi percolation in the plane.
In~\cite{Grimmett} the technique of~\cite{BKKKL} was generalized
to arbitrary product spaces satisfying some separability
condition. Finally, in~\cite{Keller-Cont-Influence} the reduction
was used to generalize several statements from the discrete cube
to the continuous cube: a lower bound on the vector of influences
obtained by Talagrand~\cite{Talagrand1}, Friedgut's theorem
characterizing functions with a low sum of
influences~\cite{Friedgut1}, and a lower bound on the size of the
boundary of monotone subsets of the discrete cube due to
Talagrand~\cite{Talagrand3} (which is weaker than
Theorem~\ref{Thm:Talagrand-Boundary} obtained by
Talagrand~\cite{Talagrand-Boundary} later).

The common feature in all these results is that they consider only
{\it monotone two-valued} functions, and the only relation between
$f$ and $g$ they obtain is a {\it lower bound} on the vector of
influences. Our paper shows that when we restrict ourselves to the
biased measure on the discrete cube (rather than the Lebesgue
measure on the continuous cube and other measure spaces considered
in the above papers), the reduction can be proved to be much more
powerful. We obtain a direct relation between the Fourier-Walsh
coefficients of $f$ and $g$, that enables us to consider {\it
general} functions (i.e., not necessarily monotone or two-valued),
and to obtain {\it upper bounds} in addition to the previously
known lower bounds.

It is interesting to find out whether the results of our paper can
be leveraged to the general reduction considering the continuous
cube with the Lebesgue measure. It seems that at least some of the
results do not generalize. For example, the characterization of
functions with a low sum of influences, that was proved by
Friedgut~\cite{Friedgut1} for the biased measure on the discrete
cube, fails for non-monotone functions on the continuous cube, as
shown by Hatami~\cite{Hatami2}.

\medskip

It should be mentioned that the reduction technique we consider is
naturally bounded. Since the transformation from $f$ cannot lead
to all possible functions $g$, the reduction does not yield a
tight result if the functions for which the respective claim is
tight for the uniform measure are outside the range of the
transformation. For example, consider the claim
\begin{equation} \label{Eq4.1}
\sum_{i=1}^n \hat f(\{i\})^2 \leq \mathbb{E}[f](1-\mathbb{E}[f]),
\end{equation}
that holds for any Boolean function w.r.t. any measure $\mu_p$ (as
follows immediately from the Parseval identity). The claim is
tight up to a constant factor for a balanced threshold function w.r.t. any measure
$\mu_p$, while the reduction yields the much weaker bound
\[
\sum_{i=1}^n \hat f(\{i\})^2 \leq \frac{1-p}{p \lfloor \log(1/p)
\rfloor} \mathbb{E}[f](1-\mathbb{E}[f]).
\]
This happens since the inequality~(\ref{Eq4.1}) is tight only for
functions having all their non-zero Fourier-Walsh coefficients on
the first level (i.e., $|S| = 1$), while the function $g$ has at
most fraction of $\frac{3p \lfloor \log(1/p) \rfloor}{1-p}$ of the
total Fourier weight on the first level by its construction (see
Remark~\ref{Rem:Lower1}).

This is the reason why Propositions~\ref{Prop:BKS-lemma}
and~\ref{Prop:Correlation} are slightly weaker than the results
obtained in~\cite{Keller-Kindler} and in~\cite{Keller-Thesis}
(respectively) by adapting the proof of the uniform measure case
to the biased measure. In both cases, the upper bound we obtain by
the reduction is weaker only for functions $f$ for which
$\sum_{i=1}^n \hat f(\{i\})^2 \gg p \lfloor \log(1/p) \rfloor $.

\medskip

The main open question left in understanding the reduction from
$\mu_p$ to $\mu_{1/2}$ is whether one can obtain an effective {\it
upper bound} on the $d$-th level Fourier-Walsh coefficients of $g$
in terms of the coefficients of $f$. Such bound for $d=1$ is given
in Remark~\ref{Rem:Lower1}, and similar (but more cumbersome)
bounds can be computed for all ``small'' values of $d$. However,
as $d$ grows, the exact formula becomes more complex and hard
to work with. A ``good'' upper bound for large values of $d$ can
lead to a simple generalization to the biased measure of the lower
bounds on the Fourier tail of Boolean and general bounded
functions~\cite{Bourgain,DFKO}, \footnote{We note that
generalizations to a biased measure of this bound for Boolean functions were obtained
by Kindler~\cite{Kindler-Thesis} and by Hatami~\cite{Hatami1} by
adapting the proof to the biased measure.} and maybe also of other
results.

\section{Acknowledgements}

We are grateful to Gil Kalai, Guy Kindler, and Elchanan Mossel for
fruitful discussions, and to an anonymous referee for suggesting a
significant simplification of the proof of
Proposition~\ref{Prop:Easy}.

\end{document}